\documentclass{amsart}

\usepackage{graphicx}
\usepackage{multicol,multirow}
\usepackage{amsmath,amssymb,amsfonts}
\usepackage{mathrsfs}
\usepackage{amsthm}
\usepackage{rotating}
\usepackage{appendix}
\usepackage[numbers]{natbib}
\usepackage[dvipsnames]{xcolor}
\usepackage{tikz}
\usepackage{float}
\usepackage{geometry}
\usepackage{txfonts}
\usepackage{epstopdf}
\usepackage{subfig}
\usepackage{booktabs}
\usepackage{fancyhdr}

\allowdisplaybreaks[2]

\newtheorem{theorem}{Theorem}[section]
\newtheorem{lemma}[theorem]{Lemma}
\newtheorem{proposition}[theorem]{Proposition}
\newtheorem{corollary}[theorem]{Corollary}

\theoremstyle{definition}

\newtheorem{example}[theorem]{Example}

\theoremstyle{remark}
\newtheorem{remark}[theorem]{Remark}
\newtheorem{question}{Question}

\newcommand{\insshuff}{\mathbin{\widehat{\text{\font\fontWCA=wncyr10 {\fontWCA x}}}}}
\newcommand{\shuff}{\mathbin{\text{\font\fontWCA=wncyr10 {\fontWCA x}}}}

\numberwithin{equation}{section}

\title[Non-associative Hopf Algebra on Dyck Paths]{Non-associative Hopf Algebra on Dyck Paths Arising from the Shuffle-insertion Product}

\thanks{Y.~Guo and H.~Li are supported by the National Natural Science Foundation of China (No.~11701339). 
N.~Bergeron is supported by an NSERC Discovery Grant (No.~06171).}

\author{Yuchen Guo}
\author{Huilan Li}
\author{Nantel Bergeron}

\address{School of Mathematics and Statistics, Shandong Normal University, 1 Daxue Road, Jinan, Shandong 250358, China}
\email{a123456gyc@163.com}
\email{lihl@sdnu.edu.cn}

\address{Department of Mathematics and Statistics, York University, 4700 Keele Street, Toronto, ON M3J 1P3, Canada}
\email{bergeron@yorku.ca}


\begin{document}

\begin{abstract}
In this paper, we define a shuffle-insertion product and a cut-box coproduct on the vector space spanned by Dyck paths, where the product is commutative but non-associative. We prove that the product and coproduct are compatible, making the space a non-associative bialgebra. This bialgebra is connected and bigraded by path degree and peak number, and hence is a non-associative Hopf algebra. Its primitive subspace is spanned by atoms, and associators of primitive elements are primitive. Furthermore, we provide closed-form formulas for its left and right antipodes. We also define the convolution product on the space of linear transformations of the non-associative Hopf algebra, which equips this space with a non-associative algebra structure. Finally, we construct a one-parameter deformation preserving the cut-box coproduct. Its zero specialization is the ordinary shuffle Hopf algebra, which is also the associated graded algebra for the atom filtration, while every nonzero specialization is isomorphic to the shuffle-insertion algebra.
\end{abstract}

\maketitle
\section{introduction}

Dyck paths are among the most classical objects in enumerative combinatorics, appearing in diverse settings such as lattice path enumeration, noncrossing partitions, cluster algebras, and representation theory. Their ubiquity stems from the rich algebraic and combinatorial structures they encode, often governed by Catalan numbers and their generalizations. Over the last two decades, Dyck paths and related families of Catalan objects have provided concrete combinatorial models for Hopf algebras, both commutative and noncommutative, with deep connections to symmetric functions, operads, and geometry \cite{2,4,8,3,14}. 

In the theory of combinatorial Hopf algebras, the algebraic structures are typically associative. A central feature of such Hopf algebras is the antipode, which serves as the Hopf-theoretic analogue of the inverse in a group. While associative Hopf algebras enjoy a rich and well-developed theory, comparatively little is known in the non-associative setting (see, e.g., \cite{15,10,16,17,11,12,13}). Recent works suggest that non-associative Hopf algebras arise naturally in renormalization, operad theory, and in the study of free magma algebras. This raises the question of whether natural combinatorial models for non-associative Hopf algebras can be found and studied systematically.

The goal of this paper is to introduce and study a new Hopf algebraic structure on Dyck paths, based on a shuffle-insertion product. This product is commutative but non-associative, and together with a natural cut-box coproduct it endows the graded vector space spanned by Dyck paths with the structure of a non-associative Hopf algebra. We refer to this algebra as the \emph{shuffle-insertion Hopf algebra} of Dyck paths.

Our starting point is the unique decomposition of a Dyck path into atoms. The usual shuffle product interleaves these atoms, while the cut-box coproduct cuts between them. We ask whether one can introduce interactions inside the atoms while retaining this coproduct. Our insertion rule does this by placing consecutive atoms of one path into specified gaps of an atom of the other; excluding gaps between adjacent $0$ and $1$ preserves the existing peaks. Peak number is additive in every product term and under the cut-box coproduct, giving a bigrading by path degree and peak number. The resulting product remains commutative and compatible with cutting, but is non-associative. This provides a concrete model in which the consequences of non-associativity can be examined through explicit path calculations. In particular, for a path with three atoms, the difference between its left and right antipodes is exactly the negative of the corresponding product associator. This associator is primitive and hence a linear combination of atoms, connecting the antipode discrepancy with primitive operations.

The main contributions of this work are as follows:  We define the shuffle-insertion product on Dyck paths and show that it is commutative but not associative (Section~\ref{sec3}).
 We construct the cut-box coproduct and prove its compatibility with the product (Sections~\ref{sec4} and~\ref{sec5}).
    We describe explicitly the left and right antipodes of the Hopf algebra, and show that they differ in general (Section~\ref{sec6}).
    We establish the bigrading by path degree and peak number, identify the primitive subspace, and relate primitive associators to the difference between the left and right antipodes (Section~\ref{sec:structure}).
     We introduce the convolution product on the space of linear transformations and prove that it is non-associative  (Section~\ref{sec7}).
     We relate the shuffle-insertion product to ordinary shuffle through a coalgebra-compatible one-parameter deformation (Section~\ref{sec8}).
     Finally, we present several open questions concerning the algebraic and combinatorial properties of this Hopf algebra, including its relation to operads, graded duality, and potential noncommutative analogues  (Section~\ref{sec9}).

This work puts the shuffle-insertion Hopf algebra of Dyck paths within the framework of non-associative Hopf algebras. It highlights Dyck paths as a good combinatorial model for exploring non-associativity in Hopf algebra theory, while also opening new directions for Dyck path combinatorics through coalgebraic and antipode structures.

\noindent\textbf{Acknowledgment.} We thank the anonymous referee for the many interesting questions and thoughtful suggestions that helped us strengthen the paper.

%%%%%%%%%%%%%%%%%%%%%%%%%%%%-2-%%%%%%%%%%%%%%%%
\section{Preliminaries}\label{sec2}

%%%%%%%%%%%%%%%%%%%%%%%%%%%%-2.1-%%%%%%%%%%%%%%%
%\subsection{Basic Definitions}\hypertarget{sec2}
Here we recall some basic definitions related to Hopf algebras and non-associative Hopf algebras; see \cite{15,10,16,17,12} for more details. 

Suppose $R$ is a commutative ring. Let $H$ be an $R$-module with a \emph{product} ${m\colon H\otimes H\rightarrow H}$, a \emph{unit} $\mu\colon R\rightarrow H$, a \emph{coproduct} $\Delta\colon H\rightarrow H\otimes H$, and a \emph{counit} $\nu\colon H\rightarrow R$, where $\otimes$ denotes $\otimes_R$. If the diagrams in Figure~\ref{fig1} commute, where $\mathrm{id}$ denotes the identity map on $H$, then $(H, m, \mu)$ is an \emph{$R$-algebra}. If the diagrams in Figure~\ref{fig2} commute, then $(H, \Delta, \nu)$ is an \emph{$R$-coalgebra}.
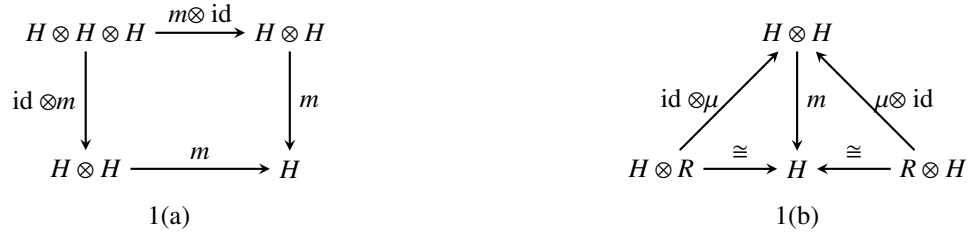
\begin{figure}[h]
\begin{minipage}[t]{0.45\textwidth}
\begin{center}
\usetikzlibrary{shapes,arrows}
\tikzstyle{arrow} = [thick,->,>=stealth]
\begin{tikzpicture}[node distance=2cm,scale=.9]
\node (A) at(0,2) [] {$H\otimes H\otimes H$};
\node (B) at(3,2) [] {$H\otimes H$};
\node (C) at(0,0) [] {$H\otimes H$};
\node (D) at(3,0) [] {$H$};
\draw [arrow] (A) -- (B)node[midway, above] {$m\otimes$ id};
\draw [arrow] (C) -- (D)node[midway, above] {$m$};
\draw [arrow] (A) -- (C)node[midway, left] {id $\otimes m$};
\draw [arrow] (B) -- (D)node[midway, right] {$m$};
\end{tikzpicture}
\end{center}
\centering{\text{1(a)}}
\end{minipage}
\hfill
\begin{minipage}[t]{0.45\textwidth}
\begin{center}
\usetikzlibrary{shapes,arrows}
\tikzstyle{arrow} = [thick,->,>=stealth]
\begin{tikzpicture}[node distance=2cm,scale=.9]]
\node (A) at(2,2) [] {$H\otimes H$};
\node (B) at(0,0) [] {$H\otimes R$};
\node (C) at(2,0) [] {$H$};
\node (D) at(4,0) [] {$R\otimes H$};
\draw [arrow] (D) -- (C)node[midway, above] {$\cong$};
\draw [arrow] (B) -- (C)node[midway, above] {$\cong$};
\draw [arrow] (A) -- (C)node[midway, right] {$m$};
\draw [arrow] (B) -- (A)node[midway, left] {id $\otimes \mu$};
\draw [arrow] (D) -- (A)node[midway, right] {$\mu\otimes$ id};
\end{tikzpicture}
\end{center}
\centering{\text{1(b)}}
\end{minipage}
\caption{Associativity and Unit\label{fig1}} 
\end{figure}
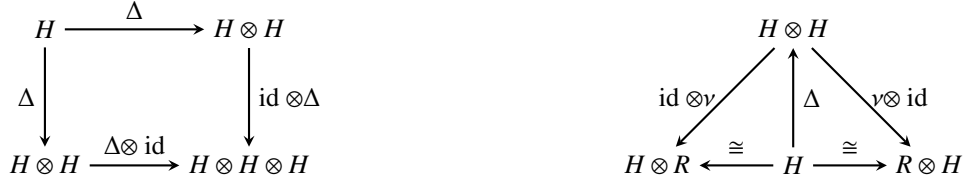
\begin{figure}[h]
\begin{minipage}[t]{0.45\textwidth}
\begin{center}
\usetikzlibrary{shapes,arrows}
\tikzstyle{arrow} = [thick,->,>=stealth]
\begin{tikzpicture}[node distance=2cm,scale=.9]]
\node (A) at(0,2) [] {$H$};
\node (B) at(3,2) [] {$H\otimes H$};
\node (C) at(0,0) [] {$H\otimes H$};
\node (D) at(3,0) [] {$H\otimes H\otimes H$};
\draw [arrow] (A) -- (B)node[midway, above] {$\Delta$};
\draw [arrow] (C) -- (D)node[midway, above] {$\Delta\otimes$ id};
\draw [arrow] (A) -- (C)node[midway, left] {$\Delta$};
\draw [arrow] (B) -- (D)node[midway, right] {id $\otimes\Delta$};
\end{tikzpicture}
\end{center}
\end{minipage}
\hfill
\begin{minipage}[t]{0.45\textwidth}
\begin{center}
\usetikzlibrary{shapes,arrows}
\tikzstyle{arrow} = [thick,->,>=stealth]
\begin{tikzpicture}[node distance=2cm,scale=.9]]
\node (A) at(2,2) [] {$H\otimes H$};
\node (B) at(0,0) [] {$H\otimes R$};
\node (C) at(2,0) [] {$H$};
\node (D) at(4,0) [] {$R\otimes H$};
\draw [arrow] (C) -- (D)node[midway, above] {$\cong$};
\draw [arrow] (C) -- (B)node[midway, above] {$\cong$};
\draw [arrow] (C) -- (A)node[midway, right] {$\Delta$};
\draw [arrow] (A) -- (B)node[midway, left] {id $\otimes \nu$};
\draw [arrow] (A) -- (D)node[midway, right] {$\nu\otimes$ id};
\end{tikzpicture}
\end{center}
\end{minipage}
\caption{Coassociativity and Counit\label{fig2}} 
\end{figure}

If $(H, m, \mu)$ is an $R$-algebra, $(H, \Delta, \nu)$ is an $R$-coalgebra, and one of the following compatibility conditions holds:
\par
(i) $\Delta$ and $\nu$ are algebra homomorphisms,\par
(ii) $m$ and $\mu$ are coalgebra homomorphisms,\\
then the algebra and coalgebra structures are said to be \emph{compatible}, and $(H,m,\mu,\Delta,\nu)$ is an \emph{$R$-bialgebra}.

If there exists a linear map $\theta\colon H\rightarrow H$ satisfying
\begin{equation*}
  m\circ(\theta\otimes \text{id})\circ\Delta=\mu\circ\nu=m\circ(\text{id}\otimes\theta)\circ\Delta,
\end{equation*}
i.e., the diagram in Figure \ref{fig3} commutes, then $\theta$ is an \emph{antipode} of $H$. A bialgebra with an antipode is a \emph{Hopf algebra}. 
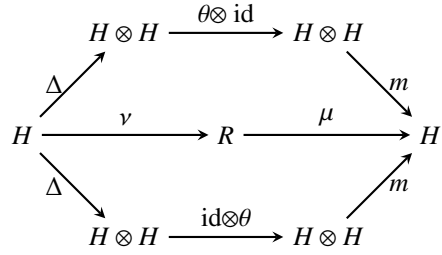
\begin{figure}[H]
	\centering
	\usetikzlibrary{shapes,arrows}
	\tikzstyle{arrow} = [thick,->,>=stealth]
	\begin{tikzpicture}[node distance=2cm,scale=.9]]
		\node (A) at(1.5,3) [] {$H\otimes H$};
		\node (B) at(4.5,3) [] {$H\otimes H$};
		\node (C) at(0,1.5) [] {$H$};
		\node (D) at(3,1.5) [] {$R$};
		\node (E) at(6,1.5) [] {$H$};
		\node (F) at(1.5,0) [] {$H\otimes H$};
		\node (G) at(4.5,0) [] {$H\otimes H$};
		\draw [arrow] (A) -- (B)node[midway, above] {$\theta\otimes$ id};
		\draw [arrow] (C) -- (A)node[midway, left] {$\Delta$};
		\draw [arrow] (B) -- (E)node[midway, right] {$m$};
		\draw [arrow] (C) -- (D)node[midway, above] {$\nu$};
		\draw [arrow] (D) -- (E)node[midway, above] {$\mu$};
		\draw [arrow] (C) -- (F)node[midway, left] {$\Delta$};
		\draw [arrow] (F) -- (G)node[midway, above] {id$\otimes \theta$};
		\draw [arrow] (G) -- (E)node[midway, right] {$m$};
	\end{tikzpicture}
	\caption{Antipode}\label{fig3}
\end{figure} 
\begin{figure}[H]
	\centering
	\usetikzlibrary{shapes,arrows}
	\tikzstyle{arrow} = [thick,->,>=stealth]
	\begin{tikzpicture}[node distance=2cm,scale=.9]
		\node (A) at(1.5,3) [] {$H\otimes H$};
		\node (B) at(4.5,3) [] {$H\otimes H$};
		\node (C) at(0,1.5) [] {$H$};
		\node (D) at(3,1.5) [] {$R$};
		\node (E) at(6,1.5) [] {$H$};
		\node (F) at(1.5,0) [] {$H\otimes H$};
		\node (G) at(4.5,0) [] {$H\otimes H$};
		\draw [arrow] (A) -- (B)node[midway, above] {$\theta_l\otimes$ id};
		\draw [arrow] (C) -- (A)node[midway, left] {$\Delta$};
		\draw [arrow] (B) -- (E)node[midway, right] {$m$};
		\draw [arrow] (C) -- (D)node[midway, above] {$\nu$};
		\draw [arrow] (D) -- (E)node[midway, above] {$\mu$};
		\draw [arrow] (C) -- (F)node[midway, left] {$\Delta$};
		\draw [arrow] (F) -- (G)node[midway, above] {id$\otimes \theta_r$};
		\draw [arrow] (G) -- (E)node[midway, right] {$m$};
	\end{tikzpicture}
	\caption{Left and Right Antipodes}\label{fig4}
\end{figure}
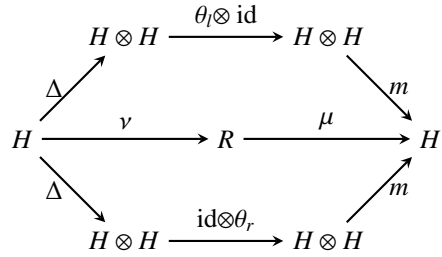

The algebra $(H,m,\mu)$ is \emph{graded} if $H=\bigoplus_{i\geq 0}H_i$ satisfies $m(H_i\otimes H_j)\subseteq H_{i+j}$ and $\mu(R)\subseteq H_0$ for any $i,j\geq0$. The coalgebra $(H,\Delta,\nu)$ is \emph{graded} if $H=\bigoplus_{i\geq 0}H_i$ satisfies $\Delta(H_n)\subseteq \bigoplus(H_i\otimes H_{n-i})$ and $\nu(H_n)=0$ for $n\geq1$. If $H=\bigoplus_{i\geq 0}H_i$ is both a graded algebra and a graded coalgebra, and a compatibility condition holds, then $H$ is a \emph{graded bialgebra}. A graded bialgebra $H=\bigoplus_{i\geq 0}H_i$ over $R$ is called \emph{connected} if $H_0=R1_H$. A \emph{graded connected bialgebra} is always a Hopf algebra, which admits a unique antipode \cite{17}.

If the product $m$ is non-associative, i.e., the diagram 1(a) in Figure~\ref{fig1} does not commute, then $(H, m, \mu)$ is called a \emph{non-associative $R$-algebra}. A \emph{non-associative bialgebra} $(H, m, \mu, \Delta, \nu)$ is defined analogously to a usual bialgebra, where the coalgebra structure remains coassociative but the algebra structure is non-associative. If there exist linear maps $\theta_l, \theta_r\colon H \to H$ satisfying
\begin{equation*}
  m\circ(\theta_l\otimes \text{id})\circ\Delta=\mu\circ\nu=m\circ(\text{id}\otimes\theta_r)\circ\Delta,
\end{equation*}
i.e., the diagram in Figure~\ref{fig4} commutes, then $\theta_l$ is called a \emph{left antipode} of $H$ and $\theta_r$ is called a \emph{right antipode} of $H$. A non-associative bialgebra equipped with left and right antipodes is called a \emph{non-associative Hopf algebra}, where the left and right antipodes do not necessarily coincide.

Throughout this paper, the term \emph{non-associative Hopf algebra} refers to the one-sided convolution identities above. In the division-based formulation of Mostovoy, P\'erez-Izquierdo, and Shestakov~\cite[Section~4.1]{11}, a non-associative Hopf algebra is equipped with bilinear left and right divisions satisfying cancellation identities. Our left and right antipodes are the linear maps specified above.

A non-associative bialgebra is called \emph{graded} if both the non-associative algebra and the coalgebra structures are graded. 
Similar to the associative case, it is not hard to show--degree by degree--that a graded connected non-associative bialgebra admits unique left and right antipodes,
and thus it is a non-associative Hopf algebra.

%%%%%%%%%%%%%%%%%%%%%%%%%%%%-2-2%%%%%%%%%%%%%%%%
%\subsection{Basic Concepts and Notations}
A \emph{Dyck path of degree $n$} is a lattice path above the line $y=x$ in the first quadrant of the $xy$-plane, from the point $(0,0)$ to the point $(n,n)$, using only steps $(0,1)$ and $(1,0)$ \cite{14}. Let $D_n$ be the set of all Dyck paths of degree $n$ with $D_0 = \{\epsilon\}$, where $\epsilon$ denotes the empty path. We represent the step $(0,1)$ by $0$ and the step $(1,0)$ by $1$. For example, the sequence ${010001100111}$ corresponds to the Dyck path of degree $6$ in Figure~\ref{fig5}.
\begin{figure}[H]
\centering
\includegraphics[width=1.5in]{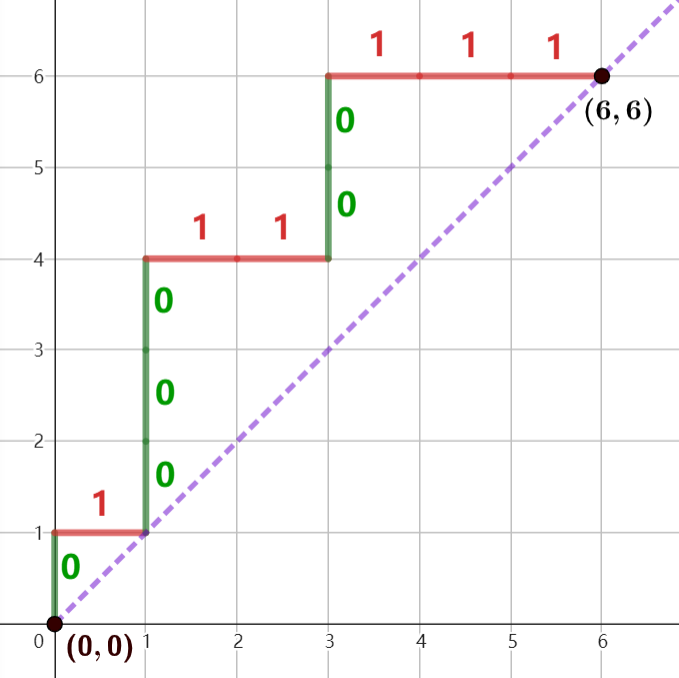}
{\caption{A Dyck Path of Degree 6}
\label{fig5}}
\end{figure}

Let $a=a_1a_2\cdots a_{2n}$ be a Dyck path of degree $n$, $a_i\in \{0,1\}$, $i=1,2,\ldots,2n$. For any $1\leq i\leq 2n-1$, the position between $a_i$ and $a_{i+1}$ is called the $i$-th \emph{gap}, the position to the left of $a_1$ is called the $0$-th  \emph{gap}, and the position to the right of $a_{2n}$ is called the $2n$-th  \emph{gap}. For example, the gaps of the Dyck path $010001100111$ are the numbers in the circles: $_{\tiny{\textcircled{0}}}0_{\tiny{\textcircled{1}}}1_{\tiny{\textcircled{2}}}0_{\tiny{\textcircled{3}}}0_{\tiny{\textcircled{4}}}0_{\tiny{\textcircled{5}}}1_
{\tiny{\textcircled{6}}}1_{\tiny{\textcircled{7}}}0_{\tiny{\textcircled{8}}}0_{\tiny{\textcircled{9}}}1_{\tiny{\textcircled{10}}}1_{\tiny{\textcircled{11}}}1_
{\tiny{\textcircled{12}}}$. Since the path remains above the line $y=x$, the number of 1's is less than or equal to the number of 0's in $a_1a_2\cdots a_{i}$ for any $0\leq i\leq 2n$. When they are equal, we call the $i$-th gap a \emph{splittable gap}. A splittable gap graphically corresponds to a position where the path touches the line $y=x$. In particular, the $0$-th and $2n$-th gaps are both splittable. The part between two adjacent splittable gaps is called an \emph{atom}. Let $A_{n}$ be the set of all atoms of degree $n$; then $A_{n}\subseteq D_{n}$ for $n\geq 1$. Let $\mathbb{A}:=\biguplus_{n\geq 1}A_n$, $\mathbb{D}:=\biguplus_{n\geq 0}D_n$, and $\mathbb{KD}=\bigoplus_{n \geq 0}\mathbb{K}D_n$, where $\mathbb{K}D_n$ is the vector space spanned by $D_n$ over a field $\mathbb{K}$ of characteristic $0$. 
 
Let $S_a=\{0,i_1,i_2,\ldots,i_{p-1},2n\}$ be the set of splittable gaps of $a$. Obviously, $i_1,i_2,\ldots,i_{p-1}$ are even numbers. 
Then $\alpha_1=a_1a_2\cdots a_{i_1}$, $\alpha_2=a_{i_1+1}a_{i_1+2}\cdots a_{i_2}$, $\ldots$, $\alpha_p=a_{i_{p-1}+1}a_{i_{p-1}+2}\cdots a_{2n}$ are the atoms of $a$, and we say that $a=\alpha_1\alpha_2\cdots \alpha_p$ is the \emph{(atom) decomposition} of $a$. Let $[k]$ denote the set $\{1,2,\ldots,k\}$ for any positive integer $k$, and let $[0]=\emptyset$.
Let $[i,j]=\{i, i+1,\ldots,j\}$ for $i\leq j$, and $[i,j]=\emptyset$ for $i>j$. We define a \emph{segment} of $a$ to be $\alpha_{[i,j]}:=\alpha_i\alpha_{i+1}\cdots\alpha_{j-1}\alpha_{j}$, where $1\leq i\leq j\leq p$.  We call the empty path $\epsilon$ and the whole path $a$ the \emph{trivial segments} of $a$. Hence, $\alpha_{[0]}=\alpha_{\emptyset}:=\epsilon$, $\alpha_{[i]}:=\alpha_1\alpha_2\cdots\alpha_i$, and $\alpha_{[p]\setminus[i]}:=\alpha_{i+1}\alpha_{i+2}\cdots\alpha_p$, for $1\leq i\leq p$. For example, the splittable gaps of the Dyck path $a=010011000111$ are $\{0,2,6,12\}$, and its atoms are $01$, $0011$, and $000111$. So $a=\alpha_1\alpha_2\alpha_3=(01)(0011)(000111)$ is the decomposition of $a$. Dyck paths $\alpha_{[2,3]}=\alpha_2\alpha_3=(0011)(000111)$, $\alpha_{[2,2]}=\alpha_2=(0011)$, and $\alpha_{[2]}=\alpha_{[1,2]}=\alpha_1\alpha_2=(01)(0011)$ are segments of $a$. A \emph{peak} of a Dyck word is an occurrence of the adjacent pair $01$. Write $\operatorname{pk}(a)$ for the number of peaks of $a$, with $\operatorname{pk}(\epsilon)=0$.

Given any atom of degree $m \geq 1$, we now define a natural \emph{pairing} of $0$'s and $1$'s. If we interpret $0$'s as open parentheses and $1$'s as close parentheses, then the well-known matching of open and close parentheses corresponds to pairing each $1$ with the nearest unmatched $0$ to its left. This results in a perfect matching, since there are always more $0$'s than $1$'s in any nonempty proper initial segment of an atom, and all $1$'s are matched with $0$'s by the end. We now place an open parenthesis at the gap immediately to the left of each $0$ and a close parenthesis at the gap immediately to the right of each $1$. For $1\leq i\leq2m-1$, we call gap $i$ an \emph{insertable gap} if there is a parenthesis at that gap; otherwise, it is a \emph{non-insertable gap}, meaning that it has a $0$ immediately to its left and a $1$ immediately to its right. Graphically, non-insertable gaps correspond to outer corners of the path.

For example, the atom ${0001100111}$ has the pairing
\[
(0{\color{blue}(0}{\color{brown}(01)}{\color{blue}1)}{\color{green}(0}{\color{orange}(01)}{\color{green}1)}1),
\]
where steps of the same color are paired. The insertable gaps are $\{1, 2, 4, 5, 6, 8, 9\}$. Note that the atom $(01)$ has no insertable gaps since its pairing is $(01)$.

%%%%%%%%%%%%%%%%%%%%%%%%%%%-3-%%%%%%%%%%%%%%%%
\section{Shuffle-insertion Product}\label{sec3}
In this section, we define the shuffle-insertion product and show that $\mathbb{KD}$ is a graded algebra with this product. 

For Dyck paths $a=\alpha_1\cdots\alpha_p$ and $b=\beta_1\cdots\beta_q$ of degrees $n$ and $m$, respectively, define their \emph{shuffle product} $a\shuff b$ as the sum of the $\binom{p+q}{p}$ interleavings of their atom sequences that preserve the relative order within each sequence. Each colored interleaving contributes once, and identical Dyck paths are collected with multiplicity. Equivalently, this product satisfies the following recursion:
\begin{equation*}
\begin{aligned}
&a\shuff  \epsilon = \epsilon\shuff  a=a,\\
&a\shuff {\color{red}b}=\alpha_1\alpha_2\cdots\alpha_p\shuff{\color{red}\beta_1\beta_2\cdots\beta_q}= \alpha_1(\alpha_2\cdots\alpha_p\shuff {\color{red}b})+{\color{red}\beta_1}(a\shuff{\color{red}\beta_2\cdots \beta_q}),
\end{aligned}
\end{equation*}
We define the \emph{unit} $\mu\colon\mathbb{K}\rightarrow\mathbb{KD}$ by $\mu(1)=\epsilon$.

It is easy to prove that the shuffle product is commutative and associative. The result of $a\shuff b$ is a sum of Dyck paths of degree $m+n$. Therefore, $(\mathbb{KD},\shuff ,\mu)$ is a graded algebra.
\begin{example}
For Dyck paths
$a=010011, {\color{red}b}={\color{red}0101}$,
\begin{equation*}
\begin{aligned} 
a\shuff {\color{red}b}
&=(01)(0011)\shuff {\color{red} (01)(01)}\\
		&=(01)(0011){\color{red} (01)(01)}+(01){\color{red}(01)}(0011){\color{red}(01)}\\
      &\ \ \ \ +(01){\color{red}(01)(01)}(0011)+{\color{red}(01)}(01)(0011){\color{red}(01)}\\
      &\ \ \ \ +{\color{red}(01)}(01){\color{red}(01)}(0011)+{\color{red}(01)(01)}(01)(0011).
\end{aligned}
\end{equation*}
The colors record which input path contributes each atom. Different colored shuffles may represent the same basis element of $\mathbb{KD}$. After forgetting the colors, their contributions are collected with multiplicities. In particular, the path $(01)(01)(0011)(01)$ occurs with coefficient $2$.
Collecting identical paths, the six contributions above become
\begin{equation*}
\begin{aligned}
a\shuff b
&=(01)(0011)(01)(01)\\
&\quad+2(01)(01)(0011)(01)\\
&\quad+3(01)(01)(01)(0011).
\end{aligned}
\end{equation*}
\end{example}
\begin{remark}\label{rem:shuff}
A useful expression for the shuffle product is
\begin{equation*}
\begin{aligned} 
\alpha_1\alpha_2\cdots\alpha_p&\shuff {\color{red}\beta_1\beta_2\cdots\beta_q}\\
        &= a{\color{red}b}+\sum_{k=1}^{\min(p,q)} \sum_{\substack{1=i_0 \le i_1 < i_2 < \cdots < i_{k+1} = p+1 \\ 1=j_0 < j_1 < \cdots < j_k \le j_{k+1} = q+1}} 
         \hskip-16pt \alpha_{[i_0,i_1-1]} {\color{red}\beta_{[j_0,j_1-1]}} \cdots
         \alpha_{[i_k,i_{k+1}-1]} {\color{red}\beta_{[j_k,j_{k+1}-1]}}\,.
\end{aligned}
\end{equation*}
Each colored shuffle of the two colored atom sequences is represented exactly once. Different colored shuffles may yield the same Dyck path after the colors are forgotten; their contributions are then combined.
\end{remark}

We now use our notion of insertable gaps to deform the shuffle product. First, we define an \emph{insertion operation} $\divideontimes$ of a Dyck path into an atom:
\begin{equation*}
\begin{aligned}	
	\divideontimes\colon \mathbb{D}\times\mathbb{A}&\rightarrow\mathbb{KD}\\
~~~~~~~~~(a,{\color{red}\beta})&\mapsto a\divideontimes{\color{red}\beta},
\end{aligned}
\end{equation*}
where $a \divideontimes {\color{red}\beta}$ denotes the sum of atoms obtained by inserting all atoms of $a$ into the insertable gaps of ${\color{red}\beta}$, while preserving the relative order of the atoms in $a$. If ${\color{red}\beta}$ has no insertable gaps, i.e., ${\color{red}\beta}$ is of degree $1$, we define $a \divideontimes {\color{red}\beta} = 0$. We say that ``$a$ is inserted into {${\color{red}\beta}$}'' to describe this operation.
\begin{example} For the Dyck path $a=010011$, its decomposition is $a=\alpha_1\alpha_2$ with $\alpha_1=(01), \alpha_2=(0(01)1)$. For the atom ${\color{red}\beta}={\color{red}001011}={\color{red}(0(01)(01)1)}$, there are three insertable gaps $1,3,5$. We insert $a$ into ${\color{red}\beta}$ and get
\begin{equation*}
\begin{aligned}
010011\divideontimes{\color{red}001011}&=
(01)(0(01)1)\divideontimes{\color{red}(0(01)(01)1)}\\
&=
{\color{red}(0}(01)(0(01)1){\color{red}(01)(01)1)}+{\color{red}(0}(01){\color{red}(01)}(0(01)1){\color{red}(01)1)}\\
&\ \ \ +{\color{red}(0}(01){\color{red}(01)(01)}(0(01)1){\color{red}1)}+{\color{red}(0(01)}(01)(0(01)1){\color{red}(01)1)}\\
&\ \ \ +{\color{red}(0(01)}(01){\color{red}(01)}(0(01)1){\color{red}1)}+{\color{red}(0(01)(01)}(01)(0(01)1){\color{red}1)}.\\
\end{aligned}
\end{equation*}
\end{example}
\begin{lemma}\label{lem:insainb}
For any Dyck path  $a$ and any atom ${\color{red}\beta}$ of degree greater than $1$,  all terms in $a \divideontimes {\color{red}\beta}$ are atoms.
\end{lemma}
\begin{proof}
In any term of $a \divideontimes {\color{red}\beta}$, all atoms of $a$ are inserted at levels strictly above the line $y = x$. Therefore, such insertions do not introduce any return to the line $y = x$, and the resulting path remains an atom.  
\end{proof}

Let $a=\alpha_1\cdots\alpha_p$ and ${\color{red}b}={\color{red}\beta_1\cdots\beta_q}$ be nonempty Dyck paths. Denote by $\operatorname{Sh}(a,{\color{red}b})$ the set of shuffles of their atom sequences, preserving the order of the atoms from each path. Throughout the construction, each atom retains its origin and index, even when two atoms represent the same Dyck path.
Fix $w\in\operatorname{Sh}(a,{\color{red}b})$. An \emph{insertion block} of $w$ is a consecutive block of atoms
\[
\gamma_1\cdots\gamma_s\eta,\qquad s\geq1,
\]
in which $\gamma_1,\ldots,\gamma_s$ all come from one input path and the receiving atom $\eta$ comes from the other.

Choose any collection of pairwise disjoint insertion blocks of $w$, including the empty collection. Replace each chosen block by
\[
(\gamma_1\cdots\gamma_s)\divideontimes\eta,
\]
leave the other atoms unchanged, and concatenate the resulting expressions, distributing over sums. All insertions are performed simultaneously using the original atoms of $w$. Thus an atom participates in at most one insertion block, and insertions cannot be nested within a single product.
Let $\tau(w)$ be the sum over all such collections. The empty collection contributes $w$ itself. 

Define the \emph{shuffle-insertion product} by
\begin{equation}\label{shuffle}
a\insshuff {\color{red}b}=\sum_{w\in\operatorname{Sh}(a,{\color{red}b})}\tau(w).
\end{equation}
Finally, forget the colors and collect identical Dyck paths with their multiplicities. Extend the product bilinearly and set $a\insshuff\epsilon=\epsilon\insshuff a=a$.
We also define the \emph{unit}  $\mu\colon\mathbb{K}\rightarrow\mathbb{KD}$ by $\mu(1)=\epsilon$.

\begin{example}
For Dyck paths $a=\alpha_1\alpha_2$ and ${\color{red}b}={\color{red}\beta_1\beta_2\beta_3\beta_4}$, consider the shuffle $w={\color{red}\beta_1}\alpha_1{\color{red}\beta_2\beta_3}\alpha_2{\color{red}\beta_4}$. The collections of disjoint insertion blocks give
\begin{equation}\label{example3}
\begin{aligned}
 \tau({\color{red}\beta_1}\alpha_1{\color{red}\beta_2\beta_3}\alpha_2{\color{red}\beta_4})
&= {\color{red}\beta_1}\alpha_1{\color{red}\beta_2\beta_3}\alpha_2{\color{red}\beta_4}
 + ({\color{red}\beta_1}\divideontimes\alpha_1){\color{red}\beta_2\beta_3}\alpha_2{\color{red}\beta_4}
 +{\color{red}\beta_1}(\alpha_1\divideontimes{\color{red}\beta_2}){\color{red}\beta_3}\alpha_2{\color{red}\beta_4}\\
&\quad
+{\color{red}\beta_1}\alpha_1{\color{red}\beta_2}({\color{red}\beta_3}\divideontimes\alpha_2){\color{red}\beta_4}
+{\color{red}\beta_1}\alpha_1({\color{red}\beta_2\beta_3}\divideontimes\alpha_2){\color{red}\beta_4}
+ {\color{red}\beta_1}\alpha_1{\color{red}\beta_2\beta_3}(\alpha_2\divideontimes{\color{red}\beta_4})\\
&\quad
+({\color{red}\beta_1}\divideontimes\alpha_1){\color{red}\beta_2}({\color{red}\beta_3}\divideontimes\alpha_2){\color{red}\beta_4}
+({\color{red}\beta_1}\divideontimes\alpha_1)({\color{red}\beta_2\beta_3}\divideontimes\alpha_2){\color{red}\beta_4}\\
&\quad
+({\color{red}\beta_1}\divideontimes\alpha_1){\color{red}\beta_2\beta_3}(\alpha_2\divideontimes{\color{red}\beta_4})+{\color{red}\beta_1}(\alpha_1\divideontimes{\color{red}\beta_2})({\color{red}\beta_3}\divideontimes\alpha_2){\color{red}\beta_4}\\
&\quad
+{\color{red}\beta_1}(\alpha_1\divideontimes{\color{red}\beta_2}){\color{red}\beta_3}(\alpha_2\divideontimes{\color{red}\beta_4}).\\
\end{aligned}
\end{equation}
We remark that any term in $\tau({\color{red}\beta_1}\alpha_1{\color{red}\beta_2\beta_3}\alpha_2{\color{red}\beta_4})$ is composed of atoms from the set
\begin{equation*}
\{ {\color{red}\beta_1}, \alpha_1, {\color{red}\beta_2}, {\color{red}\beta_3}, \alpha_2, {\color{red}\beta_4}, ({\color{red}\beta_1}\divideontimes\alpha_1), (\alpha_1\divideontimes{\color{red}\beta_2}), 
({\color{red}\beta_3}\divideontimes\alpha_2), ({\color{red}\beta_2\beta_3}\divideontimes\alpha_2), (\alpha_2\divideontimes{\color{red}\beta_4}) \},
\end{equation*}
where it is understood that $({\color{red}\beta_1}\divideontimes\alpha_1)$ denotes all terms of $({\color{red}\beta_1}\divideontimes\alpha_1)$ that are atoms, as per Lemma~\ref{lem:insainb}, and similarly for  $(\alpha_1\divideontimes{\color{red}\beta_2}), 
({\color{red}\beta_3}\divideontimes\alpha_2), ({\color{red}\beta_2\beta_3}\divideontimes\alpha_2)$ and $(\alpha_2\divideontimes{\color{red}\beta_4})$.
Note that in any term of $\tau({\color{red}\beta_1}\alpha_1{\color{red}\beta_2\beta_3}\alpha_2{\color{red}\beta_4})$, each atom from $a$ and ${\color{red}b}$ appears exactly once and in the same relative position as in $a$ and ${\color{red}b}$.

For atoms $\alpha_1=(01)$, $\alpha_2=(0(01)1)$, ${\color{red}\beta_1}={\color{red}(0(01)1)}$, ${\color{red}\beta_2}={\color{red}(01)}$, ${\color{red}\beta_3}={\color{red}(0(01)1)}$ and ${\color{red}\beta_4}={\color{red}(01)}$, we have  ${\color{red}\beta_1}\divideontimes\alpha_1=\alpha_1\divideontimes{\color{red}\beta_2}=\alpha_2\divideontimes{\color{red}\beta_4}=0$ since $\alpha_1$, ${\color{red}\beta_2}$ and ${\color{red}\beta_4}$ are of degree 1. Then
the non-zero terms in (\ref{example3}) are 
${\color{red}\beta_1}\alpha_1{\color{red}\beta_2\beta_3}\alpha_2{\color{red}\beta_4}$, ${\color{red}\beta_1}\alpha_1{\color{red}\beta_2}({\color{red}\beta_3}\divideontimes\alpha_2){\color{red}\beta_4}$, and ${\color{red}\beta_1}\alpha_1({\color{red}\beta_2\beta_3}\divideontimes\alpha_2){\color{red}\beta_4}$, where
 \begin{equation*}
\begin{aligned}
 {\color{red}\beta_3}\divideontimes\alpha_2&=(0{\color{red}(0(01)1)}(01)1)+(0(01){\color{red}(0(01)1)}1),\\
 {\color{red}\beta_2\beta_3}\divideontimes\alpha_2&=(0{\color{red}(01)(0(01)1)}(01)1)+(0{\color{red}(01)}(01){\color{red}(0(01)1)}1)+(0(01){\color{red}(01)(0(01)1)}1).\\
 \end{aligned}
\end{equation*}
 So
\begin{equation*}
\begin{aligned}
\tau({\color{red}(0(01)1)}(01){\color{red}(01)(0(01)1)}(0(01)1){\color{red}(01)})&=
{\color{red}(0(01)1)}(01){\color{red}(01)(0(01)1)}(0(01)1){\color{red}(01)}\\
&\ \ \ 
+{\color{red}(0(01)1)}(01){\color{red}(01)}(0{\color{red}(0(01)1)}(01)1){\color{red}(01)}\\
&\ \ \
+{\color{red}(0(01)1)}(01){\color{red}(01)}(0(01){\color{red}(0(01)1)}1){\color{red}(01)}\\
&\ \ \ 
+{\color{red}(0(01)1)}(01)(0{\color{red}(01)(0(01)1)}(01)1){\color{red}(01)}\\
&\ \ \ 
+{\color{red}(0(01)1)}(01)(0{\color{red}(01)}(01){\color{red}(0(01)1)}1){\color{red}(01)}\\
&\ \ \
+{\color{red}(0(01)1)}(01)(0(01){\color{red}(01)(0(01)1)}1){\color{red}(01)}.\\
\end{aligned}
\end{equation*}
\end{example}

%%%%

\begin{remark}\label{rem:atominins}
If $a$ and ${\color{red}b}$ are Dyck paths with decompositions $a = \alpha_1 \alpha_2 \cdots \alpha_p$ and ${\color{red}b} = {\color{red} \beta_1 \beta_2 \cdots \beta_q}$, then any term of $a \insshuff {\color{red}b}$ is composed of the atoms from the set
\[
  \big\{ \alpha_i, {\color{red}\beta_j}, \mathcal{A}(\alpha_{[i,i+s]} \divideontimes {\color{red}\beta_j}),\ \mathcal{A}({\color{red}\beta_{[j,j+t]}} \divideontimes \alpha_i) \ \big|\ 
  1 \le i  \le i+s \le p,\ 1 \le j \le j+t \le q \big\},
\]
where  $\mathcal{A}(\alpha_{[i,i+s]} \divideontimes {\color{red}\beta_j})$ and $\mathcal{A}({\color{red}\beta_{[j,j+t]}} \divideontimes \alpha_i)$ denote all the atoms, guaranteed by Lemma~\ref{lem:insainb}, in the respective sums.
Moreover, in any term of $a \insshuff {\color{red}b}$, the atoms $\alpha_i$ and ${\color{red}\beta_j}$ appear exactly once and in the same relative order as in $a$ and ${\color{red}b}$.
 \end{remark}
 
 \begin{example}
Let $a=\alpha_1\alpha_2$ and ${\color{red}b}={\color{red}\beta_1\beta_2}$ be Dyck paths. We use two atoms in each path to illustrate two features of the definition: a segment consisting of two atoms may be inserted into a single receiving atom, and two disjoint insertions may be performed simultaneously. We first list the six shuffles of the atoms of $a$ and ${\color{red}b}$:
\begin{equation*}
a\shuff {\color{red}b}=\alpha_1\alpha_2{\color{red} \beta_1}{\color{red} \beta_2}+\alpha_1{\color{red} \beta_1}\alpha_2{\color{red} \beta_2}+\alpha_1{\color{red} \beta_1}{\color{red} \beta_2}\alpha_2+
 {\color{red} \beta_1}\alpha_1\alpha_2{\color{red} \beta_2}+{\color{red} \beta_1}\alpha_1{\color{red} \beta_2}\alpha_2+{\color{red} \beta_1}{\color{red} \beta_2}\alpha_1\alpha_2.
\end{equation*}
Applying $\tau$ to each shuffle gives the corresponding parenthesized group in the expansion below. For example, the shuffle $\alpha_1\alpha_2{\color{red}\beta_1\beta_2}$ contributes the unchanged word, the insertion of $\alpha_2$ into ${\color{red}\beta_1}$, and the insertion of the segment $\alpha_1\alpha_2$ into ${\color{red}\beta_1}$. The alternating shuffle $\alpha_1{\color{red}\beta_1}\alpha_2{\color{red}\beta_2}$ also permits the simultaneous insertions of $\alpha_1$ into ${\color{red}\beta_1}$ and $\alpha_2$ into ${\color{red}\beta_2}$. These choices illustrate the types of atoms described in Remark~\ref{rem:atominins}.
\begin{equation*}
	\begin{aligned}
a\insshuff {\color{red}b}&=
	\tau(\alpha_1\alpha_2{\color{red} \beta_1}{\color{red} \beta_2})+\tau(\alpha_1{\color{red} \beta_1}\alpha_2{\color{red} \beta_2})+\tau(\alpha_1{\color{red} \beta_1}{\color{red} \beta_2}\alpha_2) +
 \tau({\color{red} \beta_1}\alpha_1\alpha_2{\color{red} \beta_2})+\tau({\color{red} \beta_1}\alpha_1{\color{red} \beta_2}\alpha_2)+\tau({\color{red} \beta_1}{\color{red} \beta_2}\alpha_1\alpha_2)\\
		&=
		\Big(\alpha_1\alpha_2{\color{red} \beta_1}{\color{red} \beta_2}+\alpha_1(\alpha_2\divideontimes{\color{red} \beta_1}){\color{red} \beta_2}+(\alpha_1\alpha_2\divideontimes{\color{red} \beta_1}){\color{red} \beta_2}\Big)\\
		&\ \ \ \  +
		\Big(\alpha_1{\color{red} \beta_1}\alpha_2{\color{red} \beta_2}+(\alpha_1\divideontimes{\color{red} \beta_1})\alpha_2{\color{red} \beta_2}+\alpha_1({\color{red} \beta_1}\divideontimes\alpha_2){\color{red} \beta_2}+\alpha_1{\color{red} \beta_1}(\alpha_2\divideontimes{\color{red} \beta_2})+(\alpha_1\divideontimes{\color{red} \beta_1})(\alpha_2\divideontimes{\color{red} \beta_2})\Big)\\
		&\ \ \ \  +
		\Big(\alpha_1{\color{red} \beta_1}{\color{red} \beta_2}\alpha_2+(\alpha_1\divideontimes{\color{red} \beta_1}){\color{red} \beta_2}\alpha_2
        +\alpha_1{\color{red} \beta_1}({\color{red} \beta_2}\divideontimes\alpha_2)+\alpha_1({\color{red} \beta_1}{\color{red} \beta_2}\divideontimes\alpha_2)+(\alpha_1\divideontimes{\color{red} \beta_1})({\color{red} \beta_2}\divideontimes\alpha_2)\Big)\\
		&\ \ \ \ +
		\Big({\color{red} \beta_1}\alpha_1\alpha_2{\color{red} \beta_2}+({\color{red} \beta_1}\divideontimes\alpha_1)\alpha_2{\color{red} \beta_2}+{\color{red} \beta_1}\alpha_1(\alpha_2\divideontimes{\color{red} \beta_2})+{\color{red} \beta_1}(\alpha_1\alpha_2\divideontimes{\color{red} \beta_2})+({\color{red} \beta_1}\divideontimes\alpha_1)(\alpha_2\divideontimes{\color{red} \beta_2})\Big)\\
		& \ \ \ \ +
		\Big({\color{red} \beta_1}\alpha_1{\color{red} \beta_2}\alpha_2+({\color{red} \beta_1}\divideontimes\alpha_1){\color{red} \beta_2}\alpha_2+{\color{red} \beta_1}(\alpha_1\divideontimes{\color{red} \beta_2})\alpha_2 +{\color{red} \beta_1}\alpha_1({\color{red} \beta_2}\divideontimes\alpha_2)
+({\color{red} \beta_1}\divideontimes\alpha_1)({\color{red} \beta_2}\divideontimes\alpha_2)\Big)\\
		& \ \ \ \ +
		\Big({\color{red} \beta_1}{\color{red} \beta_2}\alpha_1\alpha_2+{\color{red} \beta_1}({\color{red} \beta_2}\divideontimes\alpha_1)\alpha_2+({\color{red} \beta_1}{\color{red} \beta_2}\divideontimes\alpha_1)\alpha_2\Big).
	\end{aligned}
\end{equation*}
\end{example}

\begin{theorem}\label{theorem1}
$(\mathbb{KD},\insshuff ,\mu)$ is a commutative non-associative graded algebra.
\end{theorem}
\begin{proof}
Since the shuffle product $\shuff$ is commutative, it follows that the shuffle-insertion product $\insshuff$ is also commutative.

For the second assertion, set
\[
\alpha=(0(01)(01)1),\qquad \beta=(01),\qquad
\gamma=(0(01)1),\qquad \delta=(0(01)(01)(01)1).
\]
Directly from the definition,
\begin{align*}
\alpha\insshuff\beta&=\alpha\beta+\beta\alpha+3\delta,\\
\beta\insshuff\gamma&=\beta\gamma+\gamma\beta+2\alpha.
\end{align*}
Consider the atom $w=(0\delta(01)1)=000101011011$. For any atom, call the atoms of the Dyck path obtained by deleting its first $0$ and last $1$ its \emph{interior atoms}. In particular, $w$ has two interior atoms, $\delta$ and $(01)$.

For each of $\alpha$, $\gamma$, and $\delta$, the insertable gaps occur before, between, and after the displayed copies of $(01)$ in its interior. Inserting $r$ atoms into these gaps therefore produces an atom with respectively $2+r$, $1+r$, and $3+r$ interior atoms. Moreover, an atomic contribution to a shuffle-insertion product of two nonempty paths must come from a single insertion block containing all atoms of both inputs.

In $\delta\insshuff\gamma$, inserting $\delta$ into $\gamma$ produces $w$ and $(0(01)\delta1)\neq w$; the insertion of $\gamma$ into $\delta$ produces atoms with four interior atoms. All atomic contributions to $(\alpha\beta)\insshuff\gamma$ and $(\beta\alpha)\insshuff\gamma$ have three interior atoms. Write $[w]F$ for the coefficient of $w$ in a shuffle-insertion product $F$. Thus
\[
[w]\big((\alpha\insshuff\beta)\insshuff\gamma\big)=3.
\]
On the other hand, the atomic contributions to $\alpha\insshuff(\beta\gamma)$ and $\alpha\insshuff(\gamma\beta)$ have four interior atoms, while those to $\alpha\insshuff\alpha$ have three. Consequently,
\[
[w]\big(\alpha\insshuff(\beta\insshuff\gamma)\big)=0.
\]
The two parenthesizations therefore have different coefficients of $w$, proving that $\insshuff$ is not associative.

Obviously, we have $a\insshuff b\in\mathbb{K}D_{n+m}$ for any $a\in D_{n}$ and $b\in D_{m}$, so $(\mathbb{KD},\insshuff ,\mu)$ is a graded algebra.
Hence, $(\mathbb{KD},\insshuff ,\mu)$ is a commutative non-associative graded algebra. 
\end{proof}

\begin{proposition}\label{prop:not-free}
The algebra $(\mathbb{KD},\insshuff,\mu)$ is not free as a unital commutative non-associative algebra.
\end{proposition}
\begin{proof}
Let $w_r=(01)^r$ denote concatenation, with $w_0=\epsilon$. Since $(01)$ has no insertable gaps, only ordinary shuffles contribute, giving
\[
w_r\insshuff w_s=\binom{r+s}{r}w_{r+s}.
\]
In particular, for $x=w_1$ we have the relation
\begin{equation}\label{eq:fourth-power-relation}
((x\insshuff x)\insshuff x)\insshuff x
=(x\insshuff x)\insshuff(x\insshuff x)=24w_4.
\end{equation}
The equality between these two bracketings cannot hold for a nonscalar element of a free unital commutative non-associative algebra: by Shirshov's subalgebra theorem~\cite{Shirshov}, such an element freely generates a one-generator unital subalgebra. Since $x$ is nonscalar, the result follows.
\end{proof}

%%%%%%%%%%%%%%%%%%%%%%%%%%%-4-%%%%%%%%%%%%%%%%
\section{Cut-box Coproduct}\label{sec4}
In this section, we define the cut-box coproduct and show that $\mathbb{KD}$ is a graded coalgebra with this coproduct.

Define the \emph{cut-box coproduct} $\Delta$ on $\mathbb{KD}$ by 
$$\Delta(a)=\sum_{i=0}^{p}\alpha_{[i]}\otimes\alpha_{[p]\backslash[i]}=\sum_{i=0}^{p}\alpha_1\alpha_2\cdots\alpha_i\otimes\alpha_{i+1}\cdots\alpha_p,$$ 
where $a$ is a non-empty Dyck path of degree $n$ with decomposition $a=\alpha_1\alpha_2\cdots\alpha_p$. For $\epsilon$, $\Delta(\epsilon)=\epsilon\otimes\epsilon$.
The \emph{counit} $\nu\colon\mathbb{KD}\rightarrow\mathbb{K}$ is defined by
\begin{equation*}
		\nu(a)=
		\left\{
		\begin{aligned}
			1,\quad & a=\epsilon,\\
			0,\quad & a\neq\epsilon.
		\end{aligned}
	   \right.
\end{equation*}
\begin{example}
For the Dyck path  $a=010011001011$, its decomposition is $a=(01)(0011)(001011)$ and we have
\begin{equation*}
\begin{aligned}
	\Delta(a)
		&=
		\epsilon\otimes (01)(0011)(001011) +
		(01)\otimes(0011)(001011)\\
 &\ \ \ \ +
		(01)(0011)\otimes(001011)+
		(01)(0011)(001011)\otimes\epsilon.
\end{aligned}
\end{equation*}
\end{example} 

\begin{proposition}
$(\mathbb{KD},\Delta,\nu)$ is a graded coalgebra.
\end{proposition}
\begin{proof}
The coalgebra $(\mathbb{KD}, \Delta, \nu)$ is well-known to be the graded dual of the free graded algebra generated by the atoms. This coalgebra structure arises in several Hopf algebras in the literature; see~\cite{2,8,Loday}, for example, and references therein.
In particular, $\Delta$ is coassociative: both iterated coproducts sum over all ways of cutting the atom sequence into three consecutive, possibly empty, segments. The counit identities and compatibility with the grading follow directly from the definitions.
\end{proof}
%%%%%%%%%%%%%%%%%%%%%%%%%%%-5-%%%%%%%%%%%%%%%%
\section{A Non-associative Hopf Algebra on Dyck Paths}\label{sec5}
In this section, we prove that the shuffle-insertion product and the cut-box coproduct are compatible. Thus $(\mathbb{KD},\insshuff ,\mu,\Delta,\nu)$ is a non-associative graded bialgebra. As it is graded and connected, it is a non-associative Hopf algebra. 

\begin{theorem}\label{theorem3}
$(\mathbb{KD},\insshuff ,\mu,\Delta,\nu)$ is a non-associative graded bialgebra.
\end{theorem}

\begin{proof}
We only need to verify that the following conditions hold for any Dyck paths $a = \alpha_1 \alpha_2 \cdots \alpha_p$ and $b = \beta_1 \beta_2 \cdots \beta_q$:\par
    (i)\quad$\nu\circ\mu=\mathrm{id}_{\mathbb{K}}$,\par
    (ii)\ $~~\nu(a\insshuff b)=\nu(a)\nu(b)$,\par
	(iii)\ $~\Delta(\mu(1))=\mu(1)\otimes\mu(1)$,\par 
	(iv)\ $~\Delta(a\insshuff b)=\sum(\alpha_{[i]}\insshuff \beta_{[j]})\otimes(\alpha_{[p]\backslash[i]}\insshuff \beta_{[q]\backslash[j]})$.\par

It is easy to verify that conditions (i)--(iii) hold.  
Condition (iv) is equivalent to  
\[
\Delta(a \insshuff b) = \Delta(a) \insshuff \Delta(b),
\]
that is, the cut-box coproduct is an algebra homomorphism for the componentwise product on $\mathbb{KD}\otimes\mathbb{KD}$.

As we saw in Remark~\ref{rem:atominins}, any atom appearing in a term of $a \insshuff b$ is composed of atoms from $a$ and $b$, and they appear in the same relative order. Hence, the cut-box coproduct will contain initial segments of atoms from $a$ and $b$ in the left tensor factors, and the corresponding complementary segments in the right tensor factors. 

We now restrict our attention to those tensors in which the left factor contains exactly the first $s$ atoms of $a$ and the first $t$ atoms of $b$, and we let  
$\Delta_{s,t}(a \insshuff b)$ denote the corresponding terms in $\Delta(a \insshuff b)$.
Work with colored contributions, retaining the chosen insertion blocks and their gap choices. A cut between output atoms cannot split an insertion block, since for each fixed choice of insertion gaps, every selected block becomes one atom. It therefore separates a contribution into two independent contributions: one to the product of the prefixes $\alpha_{[s]}$ and $\beta_{[t]}$, and one to the product of their complementary suffixes. Conversely, choose a contribution to the product of the prefixes and a contribution to the product of the complementary suffixes, retaining their colored shuffles, insertion blocks, and gap choices. Concatenating these two colored shuffles and taking the union of their selected insertion blocks gives a unique contribution to the original product, with a distinguished cut between the two parts. This correspondence is reversible and preserves all insertion choices, proving the following identity with multiplicities:
\[
	\Delta_{s,t}(a\insshuff b)=(\alpha_{[s]}\insshuff \beta_{[t]})\otimes(\alpha_{[p]\backslash[s]}\insshuff \beta_{[q]\backslash[t]}).
\]
We then compute
\begin{equation*}
\begin{aligned}
 \Delta(a\insshuff b) &= \sum_{s=0}^p\sum_{t=0}^q  \Delta_{s,t}(a\insshuff b)\\
	&=  \sum_{s=0}^p\sum_{t=0}^q (\alpha_{[s]}\insshuff \beta_{[t]})\otimes(\alpha_{[p]\backslash[s]}\insshuff \beta_{[q]\backslash[t]}) \\
	&=\left(\sum_{s=0}^{p}\alpha_{[s]}\otimes \alpha_{[p]\backslash[s]}\right)\insshuff \left(\sum_{t=0}^{q}\beta_{[t]}\otimes \beta_{[q]\backslash[t]}\right)\\
	&=\Delta(a)\insshuff \Delta(b).
\end{aligned}
\end{equation*}
The third equality uses the standard componentwise extension of the product to $\mathbb{KD}\otimes\mathbb{KD}$, defined by
\[
(u\otimes v)\insshuff(u'\otimes v')=(u\insshuff u')\otimes(v\insshuff v')
\]
and extended bilinearly.
\end{proof}

\begin{example}
For Dyck paths $a=(01)(01)$, ${\color{red}b}={\color{red}(0(01)1)}$, we have
\begin{align*}
\Delta(&a\insshuff {\color{red}b})\\
&=\Delta\big((01)(01)\insshuff {\color{red}(0(01)1)}\big)\\ 
&=\Delta\Big((01)(01){\color{red}(0(01)1)}+(01){\color{red}(0}(01){\color{red}(01)1)}+(01){\color{red}(0(01)}(01){\color{red}1)}\\
&\ \ \ \ \ \ \ \ \ \ +(01){\color{red}(0(01)1)}(01)+{\color{red}(0}(01)(01){\color{red}(01)1)}+{\color{red}(0}(01){\color{red}(01)}(01){\color{red}1)}\\
&\ \ \ \ \ \ \ \ \ \ +{\color{red}(0}(01){\color{red}(01)1)}(01)+{\color{red}(0(01)}(01)(01){\color{red}1)}+{\color{red}(0(01)}(01){\color{red}1)}(01)\\
&\ \ \ \ \ \ \ \ \ \ +{\color{red}(0(01)1)}(01)(01)\Big)\\
&=\epsilon\otimes\Big((01)(01){\color{red}(0(01)1)}+(01){\color{red}(0}(01){\color{red}(01)1)}+(01){\color{red}(0(01)}(01){\color{red}1)}\\
&\ \ \ \ \ \ \ \ \ \ \ \ \ \  +(01){\color{red}(0(01)1)}(01)+{\color{red}(0}(01)(01){\color{red}(01)1)}+{\color{red}(0}(01){\color{red}(01)}(01){\color{red}1)}\\
&\ \ \ \ \ \ \ \ \ \ \ \ \ \ +{\color{red}(0}(01){\color{red}(01)1)}(01)+{\color{red}(0(01)}(01)(01){\color{red}1)}+{\color{red}(0(01)}(01){\color{red}1)}(01)\\
&\ \ \ \ \ \ \ \ \ \ \ \ \ \ +{\color{red}(0(01)1)}(01)(01)\Big)\\ 
&\ \ \ \ +(01)\otimes\Big((01){\color{red}(0(01)1)}+{\color{red}(0}(01){\color{red}(01)1)}+{\color{red}(0(01)}(01){\color{red}1)}+{\color{red}(0(01)1)}(01)\Big)\\
&\ \ \ \ +(01)(01)\otimes{\color{red}(0(01)1)}+{\color{red}(0(01)1)}\otimes(01)(01)\\
&\ \ \ \ +\Big((01){\color{red}(0(01)1)}+{\color{red}(0}(01){\color{red}(01)1)}+{\color{red}(0(01)}(01){\color{red}1)}+{\color{red}(0(01)1)}(01)\Big)\otimes(01)\\
&\ \ \ \ 
+\Big((01)(01){\color{red}(0(01)1)}+(01){\color{red}(0}(01){\color{red}(01)1)}+(01){\color{red}(0(01)}(01){\color{red}1)}\\
&\ \ \ \ \ \ \ \ +(01){\color{red}(0(01)1)}(01)+{\color{red}(0}(01)(01){\color{red}(01)1)}+{\color{red}(0}(01){\color{red}(01)}(01){\color{red}1)}\\
&\ \ \ \ \ \ \ \ +{\color{red}(0}(01){\color{red}(01)1)}(01)+{\color{red}(0(01)}(01)(01){\color{red}1)}+{\color{red}(0(01)}(01){\color{red}1)}(01)\\
&\ \ \ \ \ \ \ \ +{\color{red}(0(01)1)}(01)(01)\Big)\otimes\epsilon\\
	&=\big((\epsilon\otimes(01)(01)+(01)\otimes(01)+(01)(01)\otimes\epsilon)\big)\insshuff \big(\epsilon\otimes{\color{red}(0(01)1)}+{\color{red}(0(01)1)}\otimes\epsilon\big)\\
	&=\Delta((01)(01))\insshuff \Delta({\color{red}(0(01)1)})\\
	&=\Delta(a)\insshuff \Delta({\color{red}b}).\\
	\end{align*}
\end{example}

\begin{corollary}
$(\mathbb{KD},\insshuff ,\mu,\Delta,\nu)$ is a non-associative Hopf algebra.
\end{corollary}

\begin{proof}
From Theorem \ref{theorem3}, $(\mathbb{KD},\insshuff ,\mu,\Delta,\nu)$ is a non-associative bialgebra. Since $(\mathbb{KD},\insshuff ,\mu,\Delta,\nu)$ is graded and connected, it is a non-associative Hopf algebra.
\end{proof}

%%%%%%%%%%%%%%%%%%%%%-6-%%%%%%%%%%%%%%%
\section{Antipodes}\label{sec6}
In this section, we provide closed-form formulas for the left and right antipodes of the non-associative Hopf algebra $(\mathbb{KD}, \insshuff, \mu, \Delta, \nu)$. Takeuchi's Formula \cite[Lemma 14]{Takeuchi} for the antipode was given in a more general setting than graded connected bialgebras. For a more precise formulation in our context, see \cite[Theorem 1.1]{sagan-Benedetti}.

For any positive integer $n$, a sequence $(i_1, \ldots, i_k)$ of positive integers such that $i_1 + i_2 + \cdots + i_k = n$ is called a \emph{composition} of $n$. We denote this by $(i_1, i_2, \ldots, i_k) \models n$. For example, $(2,1) \models 3$ and $(1,2,1) \models 4$.

As $\insshuff$ is nonassociative, we adopt the following notation for Dyck paths $a$, $b$, $c$ to distinguish the two parenthesizations:
\begin{align*}
(a \insshuff b \insshuff c)_l &= (a \insshuff b) \insshuff c,\\
(a \insshuff b \insshuff c)_r &= a \insshuff (b \insshuff c).
\end{align*}
We extend these notations to denote the \emph{left} (respectively, \emph{right}) \emph{bracketing} of shuffle-insertion products involving more than three Dyck paths. Obviously, the left and right bracketings have the following properties:
\begin{align*}
(a \insshuff b \insshuff\cdots\insshuff c)_l\insshuff d&=(a \insshuff b \insshuff \cdots\insshuff c\insshuff d)_l,\\
a \insshuff (b \insshuff\cdots\insshuff c\insshuff d)_r&=(a \insshuff b\insshuff\cdots \insshuff c\insshuff d)_r.
\end{align*}

\begin{proposition}\label{antipode}
The left and right antipodes $\theta_l$, $\theta_r$ of the non-associative bialgebra $(\mathbb{KD}, \insshuff, \mu, \Delta, \nu)$ are given by
\begin{equation}\label{lantipode}
  \theta_l(a)=\sum\limits_{k=1}^p\sum\limits_{(i_1,i_2,\ldots,i_k)\models p}(-1)^{k}(\alpha_{[1,i_1]}\insshuff \alpha_{[i_1+1,i_1+i_2]}\insshuff \cdots\insshuff \alpha_{[i_1+\cdots+i_{k-1}+1,p]})_l
\end{equation}
and 
\begin{equation}\label{rantipode}
  \theta_r(a)=\sum\limits_{k=1}^p\sum\limits_{(i_1,i_2,\ldots,i_k)\models p}(-1)^{k}(\alpha_{[1,i_1]}\insshuff \alpha_{[i_1+1,i_1+i_2]}\insshuff \cdots\insshuff \alpha_{[i_1+\cdots+i_{k-1}+1,p]})_r,
\end{equation}
where $a=\alpha_1\alpha_2\cdots\alpha_p$ is a Dyck path. In particular, $\theta_l(\epsilon)=\epsilon$ and $\theta_r(\epsilon)=\epsilon$.
\end{proposition}

\begin{proof}
Regard the right-hand sides of \eqref{lantipode} and \eqref{rantipode} as definitions of candidate maps $\theta_l$ and $\theta_r$. We verify their respective antipode identities using the definitions of left and right bracketing given above. For the left antipode, we need to check that
\begin{equation*}
\insshuff \circ(\theta_l\otimes \text{id})\circ\Delta=\mu\circ\nu.
\end{equation*}
When $a=\epsilon$, we have \ $\insshuff \circ(\theta_l\otimes \text{id})\circ\Delta(\epsilon)=\epsilon=\mu\circ\nu(\epsilon).$ \ 
When $a$ is non-empty, $\mu\circ\nu(a)=0$, and we need to show that 
\begin{equation*}
\insshuff \circ(\theta_l\otimes \text{id})\circ\Delta(a)=0.
\end{equation*}
From the definitions of $\Delta$ and $\theta_l$, we have 
\begin{equation}\label{8}
 \insshuff \circ (\theta_l\otimes \text{id})\circ\Delta(a)=\sum\limits_{x=0}^p\theta_l(\alpha_{[x]})\insshuff \alpha_{[x+1,p]}.
\end{equation}
The $x$-th term $C_x$ in the sum \eqref{8} satisfies $C_0=a$. For $1\le x\le p$, substituting the candidate formula \eqref{lantipode} for the prefix $\alpha_{[x]}$, which has $x$ atoms, gives the second equality below. The third equality follows from bilinearity and the left-bracketing identity:
\begin{equation*}
\begin{aligned}
  C_x&=\theta_l(\alpha_{[x]})\insshuff \alpha_{[x+1,p]}\\
  &= \left(\sum\limits_{t=1}^x\sum\limits_{(i_1,i_2,\ldots,i_t)\models x}(-1)^{t}(\alpha_{[1,i_1]}\insshuff \alpha_{[i_1+1,i_1+i_2]}\insshuff \cdots\insshuff \alpha_{[i_1+\cdots+i_{t-1}+1,x]})_l\right)\insshuff \alpha_{[x+1,p]}\\
  &= \sum\limits_{t=1}^x\sum\limits_{(i_1,i_2,\ldots,i_t)\models x}(-1)^{t}(\alpha_{[1,i_1]}\insshuff \alpha_{[i_1+1,i_1+i_2]}\insshuff \cdots\insshuff \alpha_{[i_1+\cdots+i_{t-1}+1,x]} \insshuff \alpha_{[x+1,p]})_l.
\end{aligned}
\end{equation*}
Denote the inner sum above by $C_{x,t}$. That is,
\begin{equation*}
C_{x,t}=\sum_{(i_1,i_2,\ldots,i_t)\models x}(-1)^{t}(\alpha_{[1,i_1]}\insshuff \alpha_{[i_1+1,i_1+i_2]}\insshuff \cdots\insshuff \alpha_{[i_1+\cdots+i_{t-1}+1,x]} \insshuff \alpha_{[x+1,p]})_l.
\end{equation*}
We then have
\begin{equation*}
\begin{aligned}
 \insshuff \circ (\theta_l\otimes \text{id})\circ\Delta(a)
 &=\sum\limits_{x=0}^p\theta_l(\alpha_{[x]})\insshuff \alpha_{[x+1,p]}\\
 &=C_0 + \sum_{x=1}^p \sum_{t=1}^x C_{x,t}\\
 &=(C_0+C_{p,1}) + \sum_{t=1}^{p-1}\left( C_{p,t+1}+ \sum_{x=t}^{p-1} C_{x,t}\right).
\end{aligned}
\end{equation*}
Since $C_0+C_{p,1}=a -a = 0$, we only need to show that 
\[
C_{p,t+1}=-C_{p-1,t}-C_{p-2,t}\cdots-C_{t,t}
\]
to conclude the proof:
\begin{align*}
C_{p,t+1}
&=\sum_{(i_1,i_2,\ldots,i_{t+1})\models p}(-1)^{t+1}(\alpha_{[1,i_1]}\insshuff \alpha_{[i_1+1,i_1+i_2]}\insshuff \cdots\\
&\hspace{3.5cm}\insshuff\alpha_{[i_1+\cdots+i_{t-1}+1,i_1+\cdots+i_{t}]}\insshuff \alpha_{[i_1+\cdots+i_{t}+1,p]} \insshuff \epsilon)_l\\
&= -\sum_{y=1}^{p-t}\sum_{(i_1,i_2,\ldots,i_{t},y)\models p}(-1)^{t}(\alpha_{[1,i_1]}\insshuff \alpha_{[i_1+1,i_1+i_2]}\insshuff \cdots\\
&\hspace{4cm}\insshuff \alpha_{[i_1+\cdots+i_{t-1}+1,p-y]} \insshuff \alpha_{[i_1+\cdots+i_{t}+1, p]} )_l\\
&= -\sum_{y=1}^{p-t}\sum_{(i_1,i_2,\ldots,i_{t})\models p-y}(-1)^{t}(\alpha_{[1,i_1]}\insshuff \alpha_{[i_1+1,i_1+i_2]}\insshuff \cdots\\
&\hspace{4cm}\insshuff \alpha_{[i_1+\cdots+i_{t-1}+1,p-y]} \insshuff \alpha_{[i_1+\cdots+i_{t}+1,p]} )_l\\
&= -\sum_{y=1}^{p-t} C_{p-y,t},
\end{align*}
as needed.  
Therefore, \eqref{lantipode} is indeed a formula for the left antipode $\theta_l$. The proof of \eqref{rantipode} is similar, working on the other side and using the property of right bracketing.
\end{proof}

\begin{remark}
Although Proposition~\ref{antipode} gives closed-form formulas for the left and right antipodes, these formulas are not cancellation-free.
\end{remark}

\begin{example}
Let $x=01$, $y=0011$, and $z=001011$. We have
\[
x\insshuff y=xy+yx+2z,
\]
where the coefficient $2$ arises because inserting $x$ into either of the two insertable gaps of $y$ produces $z$. Since $xy$ has two atoms, both antipode formulas give
\begin{align*}
\theta_l(xy)=\theta_r(xy)
&=-xy+x\insshuff y\\
&=-xy+(xy+yx+2z)\\
&=yx+2z.
\end{align*}
The negative contribution $-xy$ from the composition $(2)$ cancels with the positive contribution $xy$ from the composition $(1,1)$. Thus the formulas in Proposition~\ref{antipode} are not cancellation-free.
\end{example}

\begin{corollary}
The left and right antipodes of the Hopf algebra $(\mathbb{KD}, \insshuff, \mu, \Delta, \nu)$ are not equal.
\end{corollary}

\begin{proof} 
For $a=\alpha_1\alpha_2\alpha_3$, 
\begin{equation*}
  \theta_l(a)=-\alpha_1\alpha_2\alpha_3+\alpha_1\insshuff \alpha_2\alpha_3+\alpha_1\alpha_2\insshuff \alpha_3-(\alpha_1\insshuff \alpha_2\insshuff \alpha_3)_l,
\end{equation*}
\begin{equation*}
  \theta_r(a)=-\alpha_1\alpha_2\alpha_3+\alpha_1\insshuff \alpha_2\alpha_3+\alpha_1\alpha_2\insshuff \alpha_3-(\alpha_1\insshuff \alpha_2\insshuff \alpha_3)_r.
\end{equation*}
If $\alpha_1=(0(01)(01)1)$, $\alpha_2=(01)$ and $\alpha_3=(0(01)1)$, then 
\[
((0(01)(01)1)\insshuff (01)\insshuff (0(01)1))_l\neq ((0(01)(01)1)\insshuff (01)\insshuff (0(01)1))_r
\] 
from the proof of Theorem~\ref{theorem1}. So 
\[
\theta_l((0(01)(01)1)(01)(0(01)1))\neq \theta_r((0(01)(01)1)(01)(0(01)1)).
\] 
Hence, $\theta_l\neq\theta_r$.
\end{proof}

\begin{remark}\label{rem:equal-antipodes}
Non-associativity alone does not force the left and right antipodes to be distinct. Indeed, suppose a graded connected non-associative bialgebra $H$ is both commutative and cocommutative. For every linear map $T\colon H\to H$, the commutativity and cocommutativity give
\[
m\circ(T\otimes\mathrm{id})\circ\Delta
=m\circ(\mathrm{id}\otimes T)\circ\Delta.
\]
Hence a left antipode also satisfies the right antipode identity. By uniqueness, $\theta_l=\theta_r$.

For an explicit example, consider the commutative analogue of the co-addition construction of Gerritzen and Holtkamp~\cite[Section~1]{10}. Let $A$ be the free unital commutative non-associative algebra on one generator $x$. Define $\Delta_A$ as the unital algebra homomorphism determined by
\[
\Delta_A(x)=x\otimes1+1\otimes x,
\]
and take the counit to be $1$ on the unit and zero in positive degrees. The coalgebra axioms follow from freeness by checking the corresponding algebra homomorphisms on $x$; the same argument gives cocommutativity. Grading by the number of occurrences of $x$ makes $A$ a graded connected bialgebra. Its product is non-associative, since
\[
((xx)x)x\neq(xx)(xx)
\]
are distinct bracketed monomials in $A$. Nevertheless, its left and right antipodes coincide by the preceding argument.
\end{remark}

\begin{remark}
The formulas in Proposition~\ref{antipode} can also be obtained by repeatedly expanding the antipode recursions. For the left antipode, each step cuts off a nonempty final segment, producing the left bracketing. For the right antipode, each step cuts off a nonempty initial segment, producing the right bracketing. The resulting decompositions are precisely the compositions of the number of atoms, with a decomposition into $k$ segments contributing the sign $(-1)^k$.
\end{remark}

%%%%%%%%%%%%%%%%%%%%%%%%%%%-7-%%%%%%%%%%%%%%%%
\section{Bigrading and Primitive Elements}\label{sec:structure}
We now examine two structural properties of the shuffle-insertion Hopf algebra $H=\mathbb{KD}$: its bigrading by path degree and peak number, and its primitive subspace. We also relate primitive associators to the difference between the left and right antipodes.

\begin{proposition}\label{prop:bigrading}
Let
\[
H_{n,k}=\operatorname{span}_{\mathbb{K}}\{a\in D_n:\operatorname{pk}(a)=k\}.
\]
Then $H=\mathbb{KD}=\bigoplus_{n,k\ge0}H_{n,k}$ is a bigraded non-associative Hopf algebra, with
\begin{align*}
H_{n,k}\insshuff H_{m,l}&\subseteq H_{n+m,k+l},\\
\Delta(H_{n,k})&\subseteq\bigoplus_{\substack{i+j=n\\r+s=k}}H_{i,r}\otimes H_{j,s}.
\end{align*}
Both left and right antipodes preserve each $H_{n,k}$.
\end{proposition}
\begin{proof}
Every inserted path begins with $0$ and ends with $1$. Inserting it into a gap replaces an adjacent pair $cd$ by $c\,(\text{inserted path})\,d$. Since the gap is insertable, $cd\neq01$, so no existing peak is removed. The two new adjacent pairs are $c0$ and $1d$, neither of which is $01$, so no new peak is created. Boundaries between consecutive inserted atoms, or between concatenated output atoms, are $10$ and create no peaks either. Hence every term $c$ in $a\insshuff b$ satisfies
\[
\operatorname{pk}(c)=\operatorname{pk}(a)+\operatorname{pk}(b).
\]
An internal cut between atoms separates a $10$ pair, so it preserves the total peak number of the two factors; the end cuts add an empty factor. Together with additivity of path degree, this proves the product and coproduct assertions. The unit has bidegree $(0,0)$, and the counit vanishes in positive path degree. Finally, the left and right antipode recursions preserve both degrees.
\end{proof}

\begin{proposition}\label{prop:primitives}
The primitive subspace of $H$ is
\[
\operatorname{Prim}(H):=\{u\in H:\Delta(u)=\epsilon\otimes u+u\otimes\epsilon\}
=\mathbb{K}\mathbb{A}.
\]
Moreover, for $u,v,w\in\operatorname{Prim}(H)$, the associator
\[
\operatorname{as}(u,v,w)=(u\insshuff v)\insshuff w-u\insshuff(v\insshuff w)
\]
is primitive.
\end{proposition}
\begin{proof}
Every atom is primitive by the definition of $\Delta$. For the reverse inclusion, let $u$ be primitive. The counit axioms imply that $\nu(u)=0$. For $v\in\ker\nu$, define the \emph{reduced coproduct} by
\[
\overline{\Delta}(v)=\Delta(v)-\epsilon\otimes v-v\otimes\epsilon.
\]
Thus $\overline{\Delta}(u)=0$. For a nonempty path, the reduced coproduct removes the two end cuts:
\[
\overline{\Delta}(\alpha_1\cdots\alpha_p)
=\sum_{i=1}^{p-1}\alpha_1\cdots\alpha_i\otimes\alpha_{i+1}\cdots\alpha_p.
\]
For each $p\ge2$, on the span of paths with $p$ atoms, its component with one atom on the left and $p-1$ atoms on the right sends
\[
\alpha_1\cdots\alpha_p\longmapsto\alpha_1\otimes(\alpha_2\cdots\alpha_p),
\]
which is injective because different atom sequences give different basis tensors. Components with different values of $p$ have different numbers of atoms in the right tensor factor. Hence $\overline{\Delta}(u)=0$ forces every component of $u$ with two or more atoms to vanish. Since $\nu(u)=0$, the element $u$ is a linear combination of atoms, proving the first assertion.

For primitive $u,v,w$, expand the coproduct of the associator using compatibility. All mixed tensor terms cancel, leaving
\[
\Delta(\operatorname{as}(u,v,w))
=\epsilon\otimes\operatorname{as}(u,v,w)+\operatorname{as}(u,v,w)\otimes\epsilon.
\]
This proves the second assertion.
\end{proof}

In particular, the associator gives a ternary operation on the atom space. For atoms $\alpha,\beta,\gamma$, Proposition~\ref{antipode} yields
\[
(\theta_l-\theta_r)(\alpha\beta\gamma)=-\operatorname{as}(\alpha,\beta,\gamma)\in\mathbb{K}\mathbb{A}.
\]
Thus the three-atom antipode discrepancy is a primitive operation, providing a concrete starting point for the operadic question in Section~\ref{sec9}.

%%%%%%%%%%%%%%%%%%%%%%%%%%%-8-%%%%%%%%%%%%%%%%
\section{Convolution Product}\label{sec7}
Convolution expresses the antipode identities as inverse identities in an algebra of linear maps. In the present setting, the left and right antipodes become respectively left and right convolution inverses of the identity map. We record this interpretation and explain how convolution inherits non-associativity from the underlying product.

Let $H=(\mathbb{K}\mathbb{D},\insshuff,\mu,\Delta,\nu)$, and write $m$ for the shuffle-insertion product. Let $\mathcal{L}(H,H)$ be the vector space of linear transformations of $H$. The \emph{convolution product} $\ast$ on $\mathcal{L}(H,H)$ is defined by
\[
f\ast g=m\circ(f\otimes g)\circ\Delta.
\]
Set $e=\mu\circ\nu$. Writing $\Delta(h)=\sum h_{(1)}\otimes h_{(2)}$, the unit and counit identities give
\begin{align*}
(f\ast e)(h)&=\sum f(h_{(1)})\nu(h_{(2)})=f(h),\\
(e\ast f)(h)&=\sum\nu(h_{(1)})f(h_{(2)})=f(h).
\end{align*}
Thus $e$ is the unit for convolution, and the defining antipode identities become
\[
\theta_l\ast\mathrm{id}=e=\mathrm{id}\ast\theta_r.
\]

\begin{proposition}
The convolution product on $\mathcal{L}(H,H)$ is non-associative.
\end{proposition}
\begin{proof}
For $x\in H$, define $j_x\in\mathcal{L}(H,H)$ by $j_x(h)=\nu(h)x$. Then
\begin{align*}
(j_x\ast j_y)(h)
&=\sum\nu(h_{(1)})\nu(h_{(2)})m(x,y)\\
&=\nu(h)m(x,y).
\end{align*}
Thus $j_x\ast j_y=j_{m(x,y)}$. Moreover, $j_x(\epsilon)=x$, so $x\mapsto j_x$ is an injective algebra homomorphism from $H$ into $\mathcal{L}(H,H)$. The non-associativity proved in Theorem~\ref{theorem1} therefore implies that convolution is non-associative.
\end{proof}

\begin{remark}
The same argument applies to any unital, possibly non-associative bialgebra with coassociative coproduct. Its convolution product is associative if and only if its underlying product is associative: one implication follows from the embedding above, and the other follows from coassociativity and associativity.
\end{remark}

%%%%%%%%%%%%%%%%%%%%%%%%%%%-9-%%%%%%%%%%%%%%%%
\section{A Deformation of the Shuffle Hopf Algebra}\label{sec8}
The shuffle-insertion algebra provides a concrete setting for studying coalgebra-compatible deformations of shuffle multiplication with relaxed associativity. Foissy, Patras, and Thibon~\cite{Foissy-Patras-Thibon} study associative deformations of shuffle Hopf algebras, including quasi-shuffle structures. In the present setting, filtering Dyck paths by their number of atoms recovers the ordinary shuffle algebra as the associated graded algebra. Weighting insertion terms by the decrease in atom number produces a parameterized family connecting the two products while preserving the cut-box coproduct. The explicit non-associativity example and antipode formulas then allow one to examine how insertion affects associativity and inversion, providing a starting point for comparing more general insertion rules.

Let $H=\mathbb{KD}$, and let $\ell(a)$ denote the number of atoms of a Dyck path $a$, with $\ell(\epsilon)=0$. Define the \emph{atom filtration} by $F_{-1}H=\{0\}$  and
\[
F_rH=\operatorname{span}_{\mathbb{K}}\{a\in\mathbb{D}:\ell(a)\le r\},\qquad r\ge0.
\]
Write $a\insshuff b=\sum_{c\in\mathbb{D}}m_{a,b}^{c}\,c$. Each insertion block containing $s$ inserted atoms and one receiving atom becomes a single atom, decreasing the atom count by $s$. Thus $\ell(c)\le\ell(a)+\ell(b)$ whenever $m_{a,b}^{c}\ne0$. For an indeterminate $t$, define a $\mathbb{K}[t]$-bilinear product on $H[t]=H\otimes_{\mathbb{K}}\mathbb{K}[t]$ by
\begin{equation}\label{eq:deformation-product}
a\star_t b=\sum_{c\in\mathbb{D}}m_{a,b}^{c}\,t^{\ell(a)+\ell(b)-\ell(c)}c.
\end{equation}
Extend $\mu$, $\Delta$, and $\nu$ by $\mathbb{K}[t]$-linearity.

\begin{proposition}\label{prop:shuffle-deformation}
The atom filtration makes $H$ a filtered non-associative bialgebra whose associated graded bialgebra is the ordinary shuffle Hopf algebra $(\mathbb{KD},\shuff,\mu,\Delta,\nu)$. The product $\star_t$ and the extended structure maps make $H[t]$ a graded connected, possibly non-associative bialgebra with commutative product, graded by path degree. Its specializations satisfy $\star_0=\shuff$ and $\star_1=\insshuff$. For every $\lambda\in\mathbb{K}^{\times}$, the map
\[
D_\lambda(a)=\lambda^{\ell(a)}a
\]
is a Hopf isomorphism from $(H,\star_\lambda,\mu,\Delta,\nu)$ to $(H,\insshuff,\mu,\Delta,\nu)$. In particular, every nonzero specialization is non-associative. 
\end{proposition}
\begin{proof}
The bound on atom counts gives $F_rH\insshuff F_sH\subseteq F_{r+s}H$. The contributions with no decrease in atom count are exactly those with no insertion blocks, namely the ordinary shuffles. Since deconcatenation preserves the total atom count, the coproduct also respects the filtration. This identifies the associated graded bialgebra with the ordinary shuffle Hopf algebra and proves $\star_0=\shuff$. The equality $\star_1=\insshuff$ follows directly from \eqref{eq:deformation-product}.

Over $\mathbb{K}[t,t^{-1}]$, the map $D_t(a)=t^{\ell(a)}a$ is a coalgebra automorphism preserving the unit, counit, and path degree, and
\[
D_t\bigl(a\star_t b\bigr)=D_t(a)\insshuff D_t(b).
\]
Thus $\star_t$ is commutative, unital, and compatible with the coproduct and counit. These identities hold over $\mathbb{K}[t]$ as well, since all structure constants in \eqref{eq:deformation-product} are polynomials in $t$. The path-degree grading is connected. Specializing $t$ to a nonzero $\lambda$ gives the stated isomorphism, and non-associativity follows from Theorem~\ref{theorem1}.
\end{proof}

%%%%%%%%%%%%%%%%%%%%%%%%%%%-10-%%%%%%%%%%%%%%%%
\section{Open Questions}\label{sec9}
We conclude our investigation of the non-associative Hopf algebra
\[ 
H = (\mathbb{K}\mathbb{D}, \insshuff, \mu, \Delta, \nu). 
\]
For the interested reader, we leave here a few open questions for further exploration.

\begin{question}
On the coalgebra side, $H$ is well-known to be cofree. Proposition~\ref{prop:not-free} shows that it is not free as a commutative non-associative algebra. What is an explicit generating set for $H$, and what are its defining relations?
\end{question}

\begin{question}
Is there a useful description of the graded dual of \( H \)?
\end{question}

\begin{question}
What is the role of the non-associative Hopf algebra \( H \) among graded non-associative commutative algebras? More specifically, what universal property does \( H \) satisfy in that category?
Does the discrepancy between $\theta_l$ and $\theta_r$ encode meaningful combinatorial data (such as bracketing structures in Dyck paths)? In particular, can it help determine generating functions for Dyck-path statistics?
\end{question}

\begin{question}
From an operadic perspective, the algebra structure of $H$ is of type $\mathrm{ComMag}$. In the framework of Loday~\cite[Section~2.2]{Loday}, it is natural to study the triple
\[
(\mathrm{As},\mathrm{ComMag},\mathcal{P}),
\]
where $\mathcal{P}$ is the operad governing primitive operations. The compatibility relation associated with this triple is the standard Hopf compatibility
\[
\Delta(u\insshuff v)=\Delta(u)\insshuff\Delta(v),
\]
with the componentwise product on the tensor product. Little is known about this triple with standard Hopf compatibility and its associated Hopf algebras. It would be interesting to study it with $H$ as a concrete example: Proposition~\ref{prop:primitives} identifies its primitive space and exhibits the associator as a primitive ternary operation.
\end{question}

\begin{question}
In the theory of combinatorial Hopf algebras, several commutative Hopf algebras admit natural noncommutative analogues (see~\cite{8} and references therein). Does \( H \) also have a natural noncommutative counterpart?
\end{question}
%\begin{Backmatter}

%\printaddress
%\end{Backmatter}


\begin{thebibliography}{99}

\bibitem{2} M. Aguiar, N. Bergeron, F. Sottile, 
Combinatorial Hopf algebras and generalized Dehn–Sommerville relations. 
\textit{Compositio Mathematica} \textbf{142} (2006), 1--30.

\bibitem{15} V. A. Artamonov, 
The structure of Hopf algebras. 
\textit{Journal of Mathematical Sciences} \textbf{71} (1994), 2289--2328.


\bibitem{sagan-Benedetti} C. Benedetti, B. Sagan, 
Antipodes and involutions. 
\textit{Journal of Combinatorial Theory. Series A} \textbf{148} (2017), 275--315.

\bibitem{4} N. Bergeron, C. Ceballos, 
A Hopf algebra of subword complexes. 
\textit{Advances in Mathematics} \textbf{305} (2017), 1163--1201.

\bibitem{8} N. Bergeron, R. S. Gonz\'alez D'Le\'on, S. X. Li, C. Y. A. Pang, Y. Vargas, 
Hopf algebras of parking functions and decorated planar trees. 
\textit{Advances in Applied Mathematics} \textbf{143} (2023), 102436.

\bibitem{3} N. Bergeron, H. Li, 
Algebraic structures on Grothendieck groups of a tower of algebras. 
\textit{Journal of Algebra} \textbf{321} (2009), 2068--2084.

\bibitem{7} J. Dong, H. Li, 
Hopf algebra of labeled simple graphs arising from super-shuffle product. 
\textit{Mathematics Modelling and Control} \textbf{4} (2024), 32--43.

\bibitem{Foissy-Patras-Thibon} L. Foissy, F. Patras, J.-Y. Thibon,
Deformations of shuffles and quasi-shuffles.
\textit{Annales de l'Institut Fourier} \textbf{66} (2016), 209--237.

\bibitem{10} L. Gerritzen, R. Holtkamp, 
Hopf co-addition for free magma algebras and the non-associative Hausdorff series. 
\textit{Journal of Algebra} \textbf{265} (2003), 264--284.

\bibitem{14} J. Haglund, J. Morse, M. Zabrocki, 
A compositional shuffle conjecture specifying touch points of the {D}yck path. 
\textit{Canadian Journal of Mathematics} \textbf{64-4} (2012), 822--844.

\bibitem{1} H. Hopf, 
\"Uber die Topologie der Gruppen-Mannigfaltigkeiten und ihrer Verallgemeinerungen. 
\textit{Annals of Mathematics} \textbf{42} (1941), 22--52.

\bibitem{6} M. Liu, H. Li,  
Hopf algebra on permutations arising from super-shuffle product. 
\textit{Symmetry} \textbf{13} (2021), 1010.

\bibitem{Loday} J.-L. Loday, 
Generalized bialgebras and triples of operads. 
\textit{Ast\'erisque} \textbf{320} (2008), x+116 pp.

\bibitem{16} S. V. Ludkowski, 
Non-associative algebras, rings and modules over them. 
\textit{Mathematics} \textbf{11} (2023), 1714.   

\bibitem{17} D. Manchon, 
Hopf algebras in renormalisation. 
\textit{Handbook of Algebra} \textbf{5} (2008), 365--427.

\bibitem{11} J. Mostovoy, J. M. P\'erez-Izquierdo, I. P. Shestakov, 
Hopf algebras in non-associative Lie theory. 
\textit{Bulletin of Mathematical Sciences} \textbf{4} (2014), 129--173.

\bibitem{12} J. M. P\'erez-Izquierdo, 
Algebras, hyperalgebras, non-associative bialgebras and loops. 
\textit{Advances in Mathematics} \textbf{208} (2007), 834--876.

\bibitem{Shirshov} A. I. Shirshov,
Subalgebras of free commutative and free anticommutative algebras.
\textit{Mat. Sb. (N.S.)} \textbf{34(76)} (1954), no.~1, 81--88 (in Russian).

\bibitem{18} R. Sivaraman, 
Partitions and compositions. 
\textit{International Journal of Scientific \& Engineering Research} \textbf{7} (2016), 684--691.

\bibitem{9} S. Song, H. Li, 
A Hopf algebra on (0,1)-matrices. 
\textit{Mathematics Modelling and Control} \textbf{5} (2025), 193--201.

\bibitem{Takeuchi} M. Takeuchi, 
Free Hopf algebras generated by coalgebras. 
\textit{Journal of the Mathematical Society of Japan} \textbf{23} (1971), 561--582.

\bibitem{13} U. Umirbaev, V. Kharchenko, 
Free braided non-associative Hopf algebras and Sabinin $\tau$-algebras. 
\textit{Journal of Algebra} \textbf{492} (2017), 130--156.

\bibitem{5} M. Zhao, H. Li, 
A pair of dual Hopf algebras on permutations. 
\textit{AIMS Mathematics} \textbf{6} (2021), 5106--5123.

\end{thebibliography}
\end{document}